\documentclass[reqno,12pt]{amsart}
\usepackage{amssymb}
\usepackage{amsfonts}
\usepackage{amsmath}
\usepackage{graphicx}
\usepackage{amscd,amsthm}
\usepackage{amsmath}
\usepackage{algorithm}
\usepackage{algorithmicx}
\usepackage{algpseudocode}

\newtheorem{theorem}{Theorem}
\theoremstyle{plain}
\newtheorem{acknowledgement}{Acknowledgement}

\newtheorem{corollary}{Corollary}

\newtheorem{remark}{Remark}

\numberwithin{equation}{section}

\begin{document}
\author{}
\title{}
\maketitle

\begin{center}
\thispagestyle{empty} \pagestyle{myheadings} 
\markboth{\bf Yilmaz Simsek
}{\bf Computaion negative-order Euler numbers }

\textbf{{\Large Computation Methods for combinatorial sums and Euler type numbers related to new families of numbers}}

\bigskip

\textbf{Yilmaz Simsek}\\[0pt]

Department of Mathematics, Faculty of Science University of Akdeniz TR-07058
Antalya, Turkey,

\bigskip

ysimsek@akdeniz.edu.tr\\[0pt]

\bigskip

\textbf{{\large {Abstract}}}\medskip
\end{center}

\begin{quotation}
The aim of this article is to define some new families of the special numbers.
These numbers provide some further motivation for computation of
combinatorial sums involving binomial coefficients and the Euler
kind numbers of negative order. We can show that
these numbers are related to the well-known numbers and polynomials such as
the Stirling numbers of the second kind and the central factorial numbers,
the array polynomials, the rook numbers and polynomials, the Bernstein basis
functions and others. In order to derive our new identities and relations
for these numbers, we use a technique including the generating functions and
functional equations. Finally, we not only give a computational algorithm
for these numbers, but also some numerical values of these numbers and the
Euler numbers of negative order with tables. We also give some combinatorial
interpretations of our new numbers.
\end{quotation}


\bigskip


\section{Introduction}

In this paper we define new families of numbers and polynomials. By using
these new numbers and polynomials, we give computation algorithm for the
Euler type numbers of negative orders and combinatorial sums involving
binomial coefficients. We give combinatorial interpretations of these
numbers and polynomials.

Let $a$ and $b$ be real numbers and let $\lambda $ be real or complex
numbers.

For $n$ and $k$ nonegative integers, define $y_{3}(n,k;\lambda ;a,b)$ and $%
W_{n}(\lambda )$ by means of the following generating functions,
respectively:%
\begin{equation}
	F_{y_{3}}(t,k;\lambda ;a,b)=\frac{e^{bkt}}{k!}\left( \lambda
	e^{(a-b)t}+1\right) ^{k}=\sum_{n=0}^{\infty }y_{3}(n,k;\lambda ;a,b)\frac{%
		t^{n}}{n!}.  \label{St2a1}
\end{equation}

and%
\begin{equation}
	F_{w}(t;\lambda )=\frac{1}{\lambda e^{t}+\lambda ^{-1}e^{-t}+2}%
	=\sum_{n=0}^{\infty }W_{n}(\lambda )\frac{t^{n}}{n!}.  \label{w1}
\end{equation}

Our purpose of this paper is to derive and investigate some properties of
the numbers $y_{3}(n,k;\lambda ;a,b)$, the numbers $W_{n}(\lambda )$\ and
their generating functions. We can show how the numbers $y_{3}(n,k;\lambda
;a,b)$ and $W_{n}(\lambda )$ are related to certain sequences of numbers and
polynomials such as the Stirling numbers, the central factorial numbers, the
array polynomials, the Bernstein basis functions, the first and the second
kind Euler numbers and others. We compute some values of these numbers and
polynomials, which are given by the tables. By using these numbers and their
numerical values, we also compute the first and the second kind Euler
numbers of negative order.

A summary by sections follows:

In Section 2 is a background section containing basic definitions,
identities, relations and terminology we needed. In this section, we
breiefly give certain sequences of numbers and polynomials such as the Euler
numbers, the Stirling numbers, the central factorial numbers and the array
polynomials, the Bernstein basis functions. In Section 3, we some properties
of the numbers $y_{3}(n,k;\lambda ;a,b)$ and their generating function. By
using these functions and their functional equations, we derive various
identities and relations. We also give relations between these numbers and
the well-known special numbers and polynomials. By using partial derivative
equation of the generating function, we give recurrence relations and
derivative formula of these numbers. We compute a few values of the numbers $%
y_{3}(n,k;\lambda ;a,b)$. These values are given by the tables. In Section
4, by using the numbers $y_{3}(n,k;\lambda ;a,b)$, we compute the first and
the second kind Euler numbers of negative order. We give a few values of
these numbers, which are given by the tables. We also give a computational
algorithm in order to compute the second kind Euler polynomials and numbers
of negative order. In Section 5, we give some properties of the numbers $%
W_{n}(\lambda )$. We define a new sequence of polynomials related to the
numbers $W_{n}(\lambda )$. These numbers and polynomials are related to the
first and the second kind Euler type numbers. We compute some numerical
values of these numbers. In Section 6, we give further remarks and
observations. We also give not only some combinatorial applications,
including the rook numbers and polynomials, but also combinatorial
interpretation for special values $\lambda $, $a$ and $b$ for the numbers $%
y_{3}(n,k;\lambda ;a,b)$.

\section{Background}

Throughout this paper, we use the following standard notations: 
\begin{eqnarray*}
	\mathbb{N} &=&\{1,2,3,\ldots \}, \\
	\mathbb{N}_{0} &=&\{0,1,2,3,\ldots \}=\mathbb{N}\cup \{0\},
\end{eqnarray*}%
$\mathbb{Z}$ denotes the set of integers, $\mathbb{R}$ denotes the set of
real numbers and $\mathbb{C}$ denotes the set of complex numbers.

The principal value $\ln z$ is the logarithm whose imaginary part lies in
the interval $(-\pi ,\pi ]$. Furthermore%
\begin{equation*}
	0^{n}=\left\{ 
	\begin{array}{cc}
		1, & (n=0) \\ 
		0, & (n\in \mathbb{N)}%
	\end{array}%
	\right.
\end{equation*}%
and also if $\lambda $ is a complex number, we use the following notation:%
\begin{equation*}
	\left( 
	\begin{array}{c}
		\lambda \\ 
		0%
	\end{array}%
	\right) =1\text{ and }\left( 
	\begin{array}{c}
		\lambda \\ 
		v%
	\end{array}%
	\right) =\frac{\lambda (\lambda -1)\cdots (\lambda -v+1)}{v!}=\frac{\left(
		\lambda \right) _{v}}{v!}\text{ }(n\in \mathbb{N}\text{, }\lambda \in 
	\mathbb{C)}
\end{equation*}%
(\textit{cf}. \cite{Bayad}, \cite{Comtet}, \cite{SimsekNEW}, \cite%
{SrivastavaChoi2012}). For combinatorial example, we also use the notations
of Bona \cite{Bona}. Let $\{1,2,\ldots ,n\}$ be a distinctly $n$-element
set. In work of Bona \cite[PP. 11-13.]{Bona}, the shorter notation $[n]$
stands for $(n)_{k}$: the number $n(n-1)(n-2)\cdots (n-k+1)$ of all $k$%
-element lists from $[n]$ without repetition occurs in combinatorics.

The first kind Apostol-Euler polynomials of order $k$ are defined by means
of the following generating functions:%
\begin{equation}
	F_{P1}(t,x;k,\lambda )=\left( \frac{2}{\lambda e^{t}+1}\right)
	^{k}e^{tx}=\sum_{n=0}^{\infty }E_{n}^{(k)}(x;\lambda )\frac{t^{n}}{n!}.
	\label{Cad3}
\end{equation}

We observe that%
\begin{equation*}
	E_{n}^{(k)}(\lambda )=E_{n}^{(k)}(0;\lambda )
\end{equation*}%
which denotes the first kind Apostol-Euler numbers of order $k$ (\textit{cf}%
. \cite{Grademir}-\cite{Qi}; see also the references cited in each of these
earlier works).

Substituting $k=\lambda =1$ into (\ref{Cad3}), we have the first kind Euler
numbers $E_{n}=E_{n}^{(1)}(1)$, which are defined by means of the following
generating function:%
\begin{equation*}
	\frac{2}{e^{t}+1}=\sum_{n=0}^{\infty }E_{n}\frac{t^{n}}{n!},
\end{equation*}%
where $\left\vert t\right\vert <\pi $ (\textit{cf}. \cite{Boyadzhiev}-\cite%
{Qi}; see also the references cited in each of these earlier works).

The second kind Euler numbers $E_{n}^{\ast }$ are defined by means of the
following generating function:%
\begin{equation*}
	\frac{2}{e^{t}+e^{-t}}=\sum_{n=0}^{\infty }E_{n}^{\ast }\frac{t^{n}}{n!},
\end{equation*}%
where $\left\vert t\right\vert <\frac{\pi }{2}$ (\textit{cf}. \cite{Grademir}%
-\cite{Qi}; see also the references cited in each of these earlier works).

\begin{remark}
	By using generating functions of the numbers $E_{n}$ and $E_{n}^{\ast }$, we
	can easily give a relationship between both of these numbers as follows:%
	\begin{equation*}
		E_{n}^{\ast }=2^{n}E_{n}\left( \frac{1}{2}\right) 
	\end{equation*}%
	(\textit{cf}. \cite{DSKim.JIA}, \cite{MS.KIMjnt}, \cite{KimJNT}, \cite%
	{RJMP2010}, \cite{jnt2003}; see also the references cited in each of these
	earlier works).
\end{remark}

The second kind $\lambda $-Stirling numbers $S(n,v;\lambda )$, generalized
of the classical Stirling number of the second kind, are defined by means of
the following generating function:%
\begin{equation}
	F_{S}(t,v;\lambda )=\frac{\left( \lambda e^{t}-1\right) ^{v}}{v!}%
	=\sum_{n=0}^{\infty }S(n,v;\lambda )\frac{t^{n}}{n!},  \label{SN-1}
\end{equation}%
For the numbers $S(n,v;\lambda )$, the reader should consult \cite{Luo} and 
\cite{SimsekFPTA} see also (\textit{cf}. \cite{SimsekMANISA}, \cite%
{Srivastava2011}).

From (\ref{SN-1}), we easily see that%
\begin{equation*}
	S(n,v)=S(n,v;1),
\end{equation*}%
which denotes the Stirling numbers of the second kind. These numbers are
computing by the following formula:%
\begin{equation*}
	S(n,v)=\frac{1}{v!}\sum_{j=0}^{v}\left( 
	\begin{array}{c}
		v \\ 
		j%
	\end{array}%
	\right) (-1)^{j}\left( v-j\right) ^{n}
\end{equation*}%
(\textit{cf}. \cite{Grademir}-\cite{Qi}; see also the references cited in
each of these earlier works). A recurrence relation for these numbers is
given by%
\begin{equation*}
	S(n,k)=S(n-1,k-1)+kS(n-1,k),
\end{equation*}%
with%
\begin{equation*}
	S(n,0)=0\text{ (}n\in \mathbb{N}\text{); }S(n,n)=1\text{ (}n\in \mathbb{N}%
	\text{); }S(n,1)=1\text{ (}n\in \mathbb{N}\text{)}
\end{equation*}%
and $S(n,k)=0$ ($n<k$ or $k<0$) (\textit{cf}. \cite{Grademir}-\cite{Qi}; see
also the references cited in each of these earlier works).

Let $\lambda \in \mathbb{C}$ and $k\in \mathbb{N}_{0}$. The $\lambda $-array
polynomials $S_{v}^{n}(x)$ are defined by means of the following generating
function:%
\begin{equation}
	F_{A}(t,x,k;\lambda )=\frac{1}{k!}e^{tx}\left( \lambda e^{t}-1\right)
	^{k}=\sum_{n=0}^{\infty }S_{k}^{n}(x;\lambda )\frac{t^{n}}{n!}  \label{Ary}
\end{equation}%
(\textit{cf}. \cite{SimsekFPTA}, \cite{Bayad}). Substituting $\lambda =1$
into the above equation, we have%
\begin{equation*}
	S_{k}^{n}(x)=\frac{1}{k!}\sum_{j=0}^{k}(-1)^{k-j}\left( 
	\begin{array}{c}
		k \\ 
		j%
	\end{array}%
	\right) \left( x+j\right) ^{n}
\end{equation*}%
with%
\begin{equation*}
	S_{0}^{0}(x)=S_{n}^{n}(x)=1,S_{0}^{n}(x)=x^{n}.
\end{equation*}%
If $k>n$, than%
\begin{equation*}
	S_{k}^{n}(x)=0
\end{equation*}%
(\textit{cf}. \cite{Bayad}, \cite{Chan}, \cite{SimsekFPTA}, \cite{AM2014};
see also the references cited in each of these earlier works).

The second kind central factorial numbers $T(n,k)$ are defined by means of
the following generating function:%
\begin{equation}
	F_{T}(t,k)=\frac{1}{(2k)!}\left( e^{t}+e^{-t}-2\right)
	^{k}=\sum_{n=0}^{\infty }T(n,k)\frac{t^{2n}}{(2n)!}  \label{CT-1}
\end{equation}%
(\textit{cf}. \cite{Bona}, \cite{Cigler}, \cite{Comtet}, \cite{SrivastavaLiu}%
, \cite{AM2014}; see also the references cited in each of these earlier
works).

\begin{remark}
	In the work of Alayont et al. \cite{Aloyat-1}, we observe that the central
	factorial numbers are related to the rook polynomials, which count the
	number of ways of placing non-attacking rooks on a chess board. In the work
	of Alayont and Krzywonos \cite{Alayont}, we've noticed the following elegant
	result which are related to the he central factorial numbers and the rook
	numbers:The number of ways to place $k$ rooks on a size $m$ triangle board
	in three dimensions is equal to 
	\begin{equation*}
		T(m+1,m+1-k),
	\end{equation*}
	where $0\leq k\leq m$.
\end{remark}

In \cite{SimsekNEW}, we defined the numbers $y_{1}(n,k;\lambda )$ by means
of the following generating functions:%
\begin{equation}
	F_{y_{1}}(t,k;\lambda )=\frac{1}{k!}\left( \lambda e^{t}+1\right)
	^{k}=\sum_{n=0}^{\infty }y_{1}(n,k;\lambda )\frac{t^{n}}{n!}.  \label{ay1}
\end{equation}

\begin{theorem}
	Let $n$ be a positive integer. Then we have 
	\begin{equation}
		y_{1}(n,k;\lambda )=\frac{1}{k!}\sum_{j=0}^{k}\left( 
		\begin{array}{c}
			k \\ 
			j%
		\end{array}
		\right) j^{n}\lambda ^{j}.  \label{ay2}
	\end{equation}
\end{theorem}

By substituting $\lambda =1$ into (\ref{ay2}), then we set%
\begin{equation}
	B(n,k)=k!y_{1}(n,k;1).  \label{CC2}
\end{equation}%
In \cite{golombek}, Golombek gave the following formula for (\ref{ay2}):%
\begin{equation*}
	B(n,k)=\frac{d^{n}}{dt^{n}}\left( e^{t}+1\right) ^{k}\left\vert
	_{t=0}\right. .
\end{equation*}

If we substitute $\lambda =-1$ into (\ref{ay2}), then we get the following
well-known numbers, which are so-called Stirling numbers of the second kind:%
\begin{equation*}
	S_{2}(n,k)=(-1)^{k}B(n,k).
\end{equation*}%
We \cite{SimsekNEW} gave the following \textbf{conjecture}:%
\begin{equation*}
	B(d;k)=(k^{d}+x_{1}k^{d-1}+x_{2}k^{d-2}+\cdots +x_{d-2}k^{2}+x_{d-1}k)2^{k-d}
\end{equation*}%
where $x_{1},x_{2},\ldots ,x_{d-1},d$ are positive integers.

In \cite{SimsekNEW}, we defined the numbers $y_{2}(n,k;\lambda )$ by
means of the following generating functions:%
\begin{equation}
	F_{y_{2}}(t,k;\lambda )=\frac{1}{(2k)!}\left( \lambda e^{t}+\lambda
	^{-1}e^{-t}+2\right) ^{k}=\sum_{n=0}^{\infty }y_{2}(n,k;\lambda )\frac{t^{n}%
	}{n!}.  \label{C1}
\end{equation}

Note that there is one generating function for each value of $k$.

\begin{theorem}
	\begin{equation}
		y_{2}(n,k;\lambda )=\frac{1}{\left( 2k\right) !}\sum_{j=0}^{k}\left( 
		\begin{array}{c}
			k \\ 
			j%
		\end{array}
		\right) 2^{k-j}\sum_{l=0}^{j}\left( 
		\begin{array}{c}
			j \\ 
			l%
		\end{array}
		\right) \left( 2l-j\right) ^{n}\lambda ^{2l-j}  \label{CCC3}
	\end{equation}
\end{theorem}

We will see in the following section that the above formulas provide
important insight when we are trying to compute the first and the second
kind Euler numbers of negative order.

\section{A family of new numbers $y_{3}(n,k;\protect\lambda ;a,b)$}

In this section, we give some properties of numbers $y_{3}(n,k;\lambda ;a,b)$%
. We compute a few values of these numbers. These numbers are related to the
combinatorial sums such as the stirling numbers, the array polynomials.

By using (\ref{St2a1}), we give the following explicit formula for the
numbers $y_{3}(n,k;\lambda ;a,b)$:

\begin{theorem}
	\begin{equation}
		y_{3}(n,k;\lambda ;a,b)=\frac{1}{k!}\sum_{j=0}^{k}\left( 
		\begin{array}{c}
			k \\ 
			j%
		\end{array}
		\right) \lambda ^{j}(bk+j(a-b))^{n},  \label{St2a}
	\end{equation}
\end{theorem}

\subsubsection{Setting $a=b$ in (\protect\ref{St2a}), we have%
	\protect\begin{equation*}
		y_{3}(n,k;\protect\lambda ;a,a)=\frac{\left( ak\right) ^{n}}{k!}\left( 1+%
		\protect\lambda \right) ^{k}.
		\protect\end{equation*}%
	For $k=0,1,2,3,4,5$ compute a few values of the numbers $y_{3}(n,k;\protect%
	\lambda ;a,b)$ given by Equation (\protect\ref{St2a}) as follows:%
	\protect\begin{equation*}
		y_{3}(n,0;\protect\lambda ;a,b)=0,
		\protect\end{equation*}%
	\protect\begin{equation*}
		y_{3}(n,1;\protect\lambda ;a,b)=a^{n}\protect\lambda +b^{n},
		\protect\end{equation*}%
	\protect\begin{equation*}
		y_{3}(n,2;\protect\lambda ;a,b)=\frac{\left( 2b+2(a-b)\right) ^{n}}{2}%
		\protect\lambda ^{2}+(b+a)^{n}\protect\lambda +2^{n-1}b^{n},
		\protect\end{equation*}%
	\protect\begin{equation*}
		y_{3}(n,3;\protect\lambda ;a,b)=\frac{(3b+3(a-b))^{n}}{6}\protect\lambda %
		^{3}+\frac{(3b+2(a-b))^{n}}{2}\protect\lambda ^{2}+\frac{(2b+a)^{n}}{2}%
		\protect\lambda +\frac{3^{n-1}b^{n}}{2},
		\protect\end{equation*}%
	\protect\begin{eqnarray*}
		y_{3}(n,4;\protect\lambda ;a,b) &=&\frac{(4b+4(a-b))^{n}}{24}\protect\lambda %
		^{4}+\frac{(4b+3(a-b))^{n}}{6}\protect\lambda ^{3}+\frac{(4b+2(a-b))^{n}}{4}%
		\protect\lambda ^{2} \\
		&&+\frac{(3b+a)^{n}}{6}\protect\lambda +\frac{4^{n-1}b^{n}}{6},
		\protect\end{eqnarray*}%
	\protect\begin{eqnarray*}
		y_{3}(n,5;\protect\lambda ;a,b) &=&\frac{(5b+5(a-b))^{n}}{120}\protect%
		\lambda ^{5}+\frac{(5b+4(a-b))^{n}}{24}\protect\lambda ^{4}+\frac{%
			(5b+3(a-b))^{n}}{12}\protect\lambda ^{3} \\
		&&+\frac{(5b+2(a-b))^{n}}{12}\protect\lambda ^{2}+\frac{(4b+a)^{n}}{24}%
		\protect\lambda +\frac{5^{n-1}b^{n}}{24}.
		\protect\end{eqnarray*}%
	Identities and Relations}

Here, we give some identities and relations for special values of the
numbers $y_{3}(n,k;\lambda ;a,b)$, which are given below.%
\begin{equation*}
	y_{3}(n,k;\lambda ;1,1)=\frac{k^{n}}{k!}\left( 1+\lambda \right) ^{k},
\end{equation*}%
\begin{equation*}
	y_{3}(0,k;\lambda ;1,1)=\frac{1}{k!}\left( 1+\lambda \right) ^{k},
\end{equation*}%
\begin{equation*}
	y_{3}(k,k;\lambda ;1,1)=\frac{k^{k}}{k!}\left( 1+\lambda \right) ^{k},
\end{equation*}%
\begin{equation*}
	y_{3}(k,k;1;1,1)=\frac{2^{k}k^{k}}{k!}.
\end{equation*}%
If substitute $a=1$ and $b=0$ into (\ref{St2a}), than we get%
\begin{equation*}
	y_{2}(n,k;\lambda )=y_{3}(n,k;\lambda ;1,0).
\end{equation*}%
If substitute $a=1$ and $b=-1$ into (\ref{St2a}), than we obtain%
\begin{equation*}
	y_{3}(n,k;\lambda ;1,-1)=\sum_{m=0}^{n}\left( 
	\begin{array}{c}
		n \\ 
		m%
	\end{array}%
	\right) (-1)^{n-m}k^{n-m}2^{m}y_{1}(m,k;\lambda ).
\end{equation*}%
If substitute $\lambda =-1$, $a=-1$ and $b=0$\ into (\ref{St2a}), than we
have%
\begin{equation*}
	S_{2}(n,k)=(-1)^{k+n}y_{3}(n,k;-1;-1,0).
\end{equation*}%
Now by using functional equation for the generating functions, we derive
some identities related to the numbers $y_{3}(n,k;\lambda ;a,b)$, $%
y_{1}(n,k;\lambda )$, $S_{2}(n,k)$, the Bernstein basis functions and the
array polynomials.

\begin{theorem}
	\begin{equation*}
		y_{3}(n,k;\lambda ;a,b)=\sum_{j=0}^{k}\sum_{m=0}^{n}\left( 
		\begin{array}{c}
			n \\ 
			m%
		\end{array}
		\right) a^{m}b^{n-m}y_{1}(m,j;\lambda )S_{2}(n-m,k-j).
	\end{equation*}
\end{theorem}

\begin{proof}
	By combining (\ref{SN-1}) and (\ref{ay1}) with (\ref{St2a1}), we get the
	following functional equation:%
	\begin{equation*}
		F_{y_{3}}(t,k;\lambda ;a,b)=\sum_{j=0}^{k}F_{y_{1}}(at,j;\lambda
		)F_{S}(bt,k-j;1).
	\end{equation*}%
	From this equation, we get%
	\begin{equation*}
		\sum_{n=0}^{\infty }y_{3}(n,k;\lambda ;a,b)\frac{t^{n}}{n!}%
		=\sum_{j=0}^{k}\sum_{n=0}^{\infty }a^{n}y_{1}(n,j;\lambda )\frac{t^{n}}{n!}%
		\sum_{n=0}^{\infty }b^{n}S_{2}(n,k-j)\frac{t^{n}}{n!}.
	\end{equation*}%
	Therefore%
	\begin{equation*}
		\sum_{n=0}^{\infty }y_{3}(n,k;\lambda ;a,b)\frac{t^{n}}{n!}%
		=\sum_{n=0}^{\infty }\sum_{j=0}^{k}\sum_{m=0}^{n}\left(
		\begin{array}{c}
			n \\
			m%
		\end{array}%
		\right) a^{m}b^{n-m}y_{1}(m,j;\lambda )S_{2}(n-m,k-j)\frac{t^{n}}{n!}.
	\end{equation*}%
	Comparing the coefficients of $\frac{t^{n}}{n!}$ on both sides of the above
	equation, we arrive at the desired result.
\end{proof}

A relation between the numbers $y_{1}(n,k;\lambda )$ and $y_{3}(n,k;\lambda
;a,b)$ is given by the following theorem:

\begin{theorem}
	\begin{equation*}
		y_{3}(n,k;\lambda ;a,b)=\sum_{j=0}^{n}\left( 
		\begin{array}{c}
			n \\ 
			j%
		\end{array}
		\right) (a-b)^{j}(bk)^{n-j}y_{1}(j,k;\lambda ).
	\end{equation*}
\end{theorem}

\begin{proof}
	By using (\ref{ay1}) and (\ref{St2a1}), we obtain the following functional
	equation:%
	\begin{equation}
		F_{y_{3}}(t,k;\lambda ;a,b)=e^{bkt}F_{y_{1}}((a-b)t,k;\lambda ).  \label{BER}
	\end{equation}%
	From this equation, we get%
	\begin{equation*}
		\sum_{n=0}^{\infty }y_{3}(n,k;\lambda ;a,b)\frac{t^{n}}{n!}%
		=\sum_{n=0}^{\infty }(bk)^{n}\frac{t^{n}}{n!}\sum_{n=0}^{\infty
		}(a-b)^{n}y_{1}(n,k;\lambda )\frac{t^{n}}{n!}.
	\end{equation*}%
	Therefore%
	\begin{equation*}
		\sum_{n=0}^{\infty }y_{3}(n,k;\lambda ;a,b)\frac{t^{n}}{n!}%
		=\sum_{n=0}^{\infty }\sum_{j=0}^{n}\left(
		\begin{array}{c}
			n \\
			j%
		\end{array}%
		\right) (a-b)^{j}(bk)^{n-j}y_{1}(j,k;\lambda )\frac{t^{n}}{n!}.
	\end{equation*}%
	Comparing the coefficients of $\frac{t^{n}}{n!}$ on both sides of the above
	equation, we arrive at the desired result.
\end{proof}

Combining (\ref{Ary}) and (\ref{St2a1}), we obtain the following functional
equation:%
\begin{equation*}
	F_{y_{3}}(t,k;-1;a,b)=(-1)^{k}F_{A}\left( (a-b)t,\frac{b}{a-b},k\right) .
\end{equation*}

By using this equation, a relation between the array polynomials and the
numbers $y_{3}(n,k;-1;a,b)$ by the following theorem:

\begin{theorem}
	\begin{equation*}
		y_{3}(n,k;-1;a,b)=(-1)^{k}(a-b)^{n}S_{v}^{n}\left( \frac{b}{a-b}\right) .
	\end{equation*}
\end{theorem}

By replacing $\lambda $ by $-\lambda ^{2}$\ in (\ref{Ary}), using (\ref{ay1}%
) and (\ref{St2a1}), we also have the following functional equation:%
\begin{equation*}
	F_{y_{3}}(2t,k;-\lambda ^{2};a,b)=(-1)^{k}k!F_{y_{1}}((a-b)t,k;\lambda
	)F_{A}\left( (a-b)t,\frac{2b}{a-b},k;\lambda \right) .
\end{equation*}%
By using this equation with (\ref{Ary}), (\ref{ay1}) and (\ref{St2a1}), we
get the following theorem:

\begin{theorem}
	\begin{equation*}
		y_{3}(n,k;-\lambda ^{2};a,b)=(-1)^{k}k!\left( \frac{a-b}{2}\right)
		^{n}\sum_{j=0}^{n}\left( 
		\begin{array}{c}
			n \\ 
			j%
		\end{array}
		\right) y_{1}(n,j;\lambda )S_{k}^{n-j}\left( \frac{b}{a-b}\right) .
	\end{equation*}
\end{theorem}

By using (\ref{CT-1}) and (\ref{St2a1}), we give the following functional
equation:%
\begin{equation*}
	F_{y_{3}}(t,k;1;1,-1)=\frac{1}{k!}\sum_{j=0}^{k}\left( 
	\begin{array}{c}
		k \\ 
		j%
	\end{array}%
	\right) (-2)^{k-j}\left( 2j\right) !F_{T}(t,j).
\end{equation*}%
By this equation, we get a relation between the central factorial numbers
and the numbers $y_{3}(n,k;1;1,-1)$ by the following theorem:

\begin{theorem}
	If $n$ is an even integer, we have 
	\begin{equation*}
		y_{3}(n,k;1;1,-1)=\frac{1}{k!}\sum_{j=0}^{k}\left( 
		\begin{array}{c}
			k \\ 
			j%
		\end{array}
		\right) (-2)^{k-j}\left( 2j\right) !T(n,j).
	\end{equation*}
	If $n$ is an odd integer, we have 
	\begin{equation*}
		y_{3}(n,k;1;1,-1)=0.
	\end{equation*}
\end{theorem}

Substituting $b=x$ and $a=1$ into (\ref{St2a}), we give relationships
between the numbers $y_{3}(n,k;\lambda ;a,b)$, the numbers $%
y_{1}(n-m,k;\lambda )$\ and the Bernstein basis functions $B_{k}^{n}(x)$ by
the following corollary:

\begin{corollary}
	We have 
	\begin{equation*}
		y_{3}(n,k;\lambda ;1,x)=\frac{1}{k!}\sum_{j=0}^{k}\left( 
		\begin{array}{c}
			k \\ 
			j%
		\end{array}%
		\right) \lambda ^{j}\sum_{m=0}^{n}k^{m}j^{n-m}B_{m}^{n}\left( x\right)
	\end{equation*}%
	and 
	\begin{equation}
		y_{3}(n,k;\lambda ;1,x)=\sum_{m=0}^{n}k^{m}B_{m}^{n}\left( x\right)
		y_{1}(n-m,k;\lambda ),  \label{BE-1}
	\end{equation}%
	where $B_{m}^{n}\left( x\right) $\ denotes the Bernstein basis functions:%
	\begin{equation*}
		B_{m}^{n}\left( x\right) =\left( 
		\begin{array}{c}
			n \\ 
			k%
		\end{array}%
		\right) x^{k}\left( 1-x\right) ^{n-k}
	\end{equation*}%
	(\textit{cf}. \cite{Lorenz}, \cite{mmas2015}; see also the references cited
	in each of these earlier works).
\end{corollary}

Integrating both sides of Equation (\ref{BE-1}) from $0$ to $1$, and using 
\begin{equation*}
	\int_{0}^{1}B_{m}^{n}\left( x\right) dx=\frac{1}{n+1}
\end{equation*}%
(\textit{cf}. \cite{Lorenz}, \cite{mmas2015}; see also the references cited
in each of these earlier works), we get the following theorem:

\begin{theorem}
	\begin{equation*}
		\int_{0}^{1}y_{3}(n,k;\lambda ;1,x)dx=\frac{1}{n+1}
		\sum_{m=0}^{n}k^{m}y_{1}(n-m,k;\lambda )
	\end{equation*}
	or 
	\begin{equation*}
		\int_{0}^{1}y_{3}(n,k;\lambda ;1,x)dx=\frac{1}{\left( n+1\right) k!}
		\sum_{j=0}^{k}\left( 
		\begin{array}{c}
			k \\ 
			j%
		\end{array}
		\right) j^{n}\lambda ^{j}\sum_{m=0}^{n}k^{m}j^{-m}.
	\end{equation*}
\end{theorem}

\subsection{Recurrence relation for the numbers $y_{3}(n,k;\protect\lambda %
	;a,b)$}

Here using derivative operator to generating function for the numbers $%
y_{3}(n,k;\lambda ;a,b)$, we give a recurrence relation for these numbers.

Taking derivative of (\ref{St2a1}), with respect to $t$, we obtain the
following partial differential equation:

\begin{equation}
	\frac{\partial }{\partial t}F_{y_{3}}(t,k;\lambda
	;a,b)=bkF_{y_{3}}(t,k;\lambda ;a,b)+(a-b)\lambda F_{y_{3}}(t,k-1;\lambda
	;a,b).  \label{St2a2}
\end{equation}%
Combining this equation with (\ref{St2a1}), we derive a recurrence relation
for the numbers $y_{3}(n,k;\lambda ;a,b)$ by the following theorem:

\begin{theorem}
	Let $k$ be a positive integer. Then we have 
	\begin{equation*}
		y_{3}(n+1,k;\lambda ;a,b)=bky_{3}(n,k;\lambda ;a,b)+(a-b)y_{3}(n,k-1;\lambda
		;a,b).
	\end{equation*}
\end{theorem}

\begin{proof}
	By using (\ref{St2a2}) and (\ref{St2a1}), we get%
	\begin{eqnarray*}
		&&\sum_{n=1}^{\infty }y_{3}(n,k;\lambda ;a,b)\frac{t^{n-1}}{\left(
			n-1\right) !} \\
		&=&bk\sum_{n=0}^{\infty }y_{3}(n,k;\lambda ;a,b)\frac{t^{n}}{n!}%
		+(a-b)\sum_{n=0}^{\infty }y_{3}(n,k-1;\lambda ;a,b)\frac{t^{n}}{n!}.
	\end{eqnarray*}%
	Therefore%
	\begin{eqnarray*}
		&&\sum_{n=0}^{\infty }y_{3}(n+1,k;\lambda ;a,b)\frac{t^{n}}{n!} \\
		&=&bk\sum_{n=0}^{\infty }y_{3}(n,k;\lambda ;a,b)\frac{t^{n}}{n!}%
		+(a-b)\sum_{n=0}^{\infty }y_{3}(n,k-1;\lambda ;a,b)\frac{t^{n}}{n!}.
	\end{eqnarray*}%
	Comparing the coefficients of $\frac{t^{n}}{n!}$ on both sides of the above
	equation, we arrive at the desired result.
\end{proof}

Taking derivative of (\ref{St2a1}), with respect to $\lambda $, we obtain
the following partial differential equation:%
\begin{equation*}
	\frac{\partial }{\partial \lambda }F_{y_{3}}(t,k;\lambda ;a,b)=e^{\left(
		a-b\right) t}F_{y_{3}}(t,k-1;\lambda ;a,b).
\end{equation*}%
By using the same processes in the above theorem, we obtain the following
theorem:%
\begin{equation*}
	\frac{\partial }{\partial \lambda }y_{3}(n,k;\lambda
	;a,b)=\sum_{m=0}^{n}\left( 
	\begin{array}{c}
		n \\ 
		m%
	\end{array}%
	\right) \left( a-b\right) ^{n-m}y_{3}(m,k-1;\lambda ;a,b).
\end{equation*}

\section{Computation of the Euler numbers of negative order}

In this section, by using the numbers $y_{3}(n,k;\lambda ;a,b)$, we compute
values of the second kind Apostol type Euler polynomials of negative order.
We also give a computation algorithm for computing the values of these
polynomials.

We \cite{SimsekNEW} defined the second kind Apostol type Euler polynomials
of order $k$, $E_{n}^{\ast (k)}(x;\lambda )$ by means of the following
generating functions:%
\begin{equation}
	F_{P}(t,x;k,\lambda )=\left( \frac{2}{\lambda e^{t}+\lambda ^{-1}e^{-t}}%
	\right) ^{k}e^{tx}=\sum_{n=0}^{\infty }E_{n}^{\ast (k)}(x;\lambda )\frac{%
		t^{n}}{n!}.  \label{Eul.2}
\end{equation}%
We observe that%
\begin{equation*}
	E_{n}^{\ast (k)}(\lambda )=E_{n}^{\ast (k)}(0;\lambda )
\end{equation*}%
denotes the second kind Apostol type Euler numbers of order $k$.

We can give a motivation on (\ref{Eul.2}) as follows:%
\begin{equation*}
	F_{P}(t,x;k,\lambda )=F_{H}\left( t,\frac{x+k}{2};k,-\lambda ^{-2}\right) ,
\end{equation*}%
where%
\begin{equation*}
	F_{H}\left( t,x;k,u\right) =\frac{1-u}{e^{t}-u}e^{tx}=\sum_{n=0}^{\infty
	}H_{n}^{(k)}(x;u)\frac{t^{n}}{n!}
\end{equation*}%
$u\neq 1$ and $H_{n}^{(k)}(x;u)$ denotes the Frobenius-Euler polynomials of
higher order. From the above functional equation, we get%
\begin{equation*}
	E_{n}^{\ast (k)}(x;\lambda )=\frac{2^{n}}{\lambda ^{k-2}(\lambda ^{2}+1)}%
	H_{n}^{(k)}\left( \frac{x+k}{2};-\lambda ^{-2}\right) .
\end{equation*}%
These numbers are also related to the twisted Euler numbers and polynomials (%
\textit{cf}. \cite{DSKim.JIA}, \cite{KimJNT}, \cite{jnt2003}).

The first kind Apostol-Euler numbers of order $-k$ are defined by means of
the following generating functions:%
\begin{equation}
	\left( \frac{\lambda e^{t}+1}{2}\right) ^{k}=\sum_{n=0}^{\infty
	}E_{n}^{(-k)}(\lambda )\frac{t^{n}}{n!}  \label{ae-1}
\end{equation}%
(\textit{cf}. \cite{SimsekNEW}, \cite{SrivastavaChoi2012}; see also the
references cited in each of these earlier works). The second kind Apostol
type Euler numbers of order $-k$ are defined by means of the following
generating functions:%
\begin{equation}
	F_{N}(t;-k,\lambda )=\left( \frac{\lambda e^{t}+\lambda ^{-1}e^{-t}}{2}%
	\right) ^{k}=\sum_{n=0}^{\infty }E_{n}^{\ast (-k)}(\lambda )\frac{t^{n}}{n!}
	\label{Cac3}
\end{equation}%
(\textit{cf}. \cite{SimsekNEW}; see also the references cited in each of
these earlier works)

\begin{theorem}
	(\cite{SimsekNEW}) Let $k$ be nonnegative integer. Then we have 
	\begin{equation}
		E_{n}^{(-k)}(\lambda )=k!2^{-k}y_{1}(n,k;\lambda ).  \label{Cab3}
	\end{equation}
\end{theorem}

\begin{remark}
	Byrd \cite{Byrd} and Liu \cite{Liu1} also gave a formula for the numbers $%
	E_{n}^{(-k)}$. In (\cite{SimsekNEW}), we computed a few values of the first
	kind Euler numbers of order $-k$ by the following formula 
	\begin{equation}
		E_{n}^{(-k)}=2^{-k}\sum_{j=0}^{k}\left( 
		\begin{array}{c}
			k \\ 
			j%
		\end{array}%
		\right) j^{n}  \label{Caa3}
	\end{equation}%
	as follows: for $n=1,2,\ldots ,9$ and $k=0,-1,-2,\ldots ,-9$, we compute a
	few values of the numbers $E_{n}^{(-k)}$, given by the above relations, as
	follows: 
	\begin{equation*}
		\begin{tabular}{lllllllllll}
			$n\backslash k$ & $0$ & $-1$ & $-2$ & $-3$ & $-4$ & $-5$ & $-6$ & $-7$ & $-8$
			& $-9\cdots $ \\ 
			$0$ & $0$ & $\frac{1}{2}$ & $\frac{3}{4}$ & $\frac{7}{8}$ & $\frac{15}{16}$
			& $\frac{33}{32}$ & $\frac{33}{64}$ & $\frac{81}{64}$ & $\cdots $ & $\cdots $
			\\ 
			$1$ & $0$ & $\frac{1}{2}$ & $1$ & $\frac{3}{2}$ & $2$ & $\frac{5}{2}$ & $3$
			& $\frac{7}{2}$ & $4$ & $\frac{9}{2}\cdots $ \\ 
			$2$ & $0$ & $\frac{1}{2}$ & $\frac{3}{2}$ & $3$ & $5$ & $\frac{15}{2}$ & $%
			\frac{21}{2}$ & $14$ & $18$ & $\frac{45}{2}\cdots $ \\ 
			$3$ & $0$ & $\frac{1}{2}$ & $\frac{5}{2}$ & $\frac{27}{4}$ & $14$ & $25$ & $%
			\frac{81}{2}$ & $\frac{245}{4}$ & $88$ & $\frac{243}{2}\cdots $ \\ 
			$4$ & $0$ & $\frac{1}{2}$ & $\frac{9}{2}$ & $\frac{33}{2}$ & $\frac{85}{2}$
			& $90$ & $168$ & $287$ & $459$ & $\frac{1395}{2}\cdots $ \\ 
			$5$ & $0$ & $\frac{1}{2}$ & $\frac{17}{2}$ & $\frac{171}{4}$ & $137$ & $%
			\frac{1375}{4}$ & $738$ & $1421$ & $2524$ & $4212\cdots $ \\ 
			$6$ & $0$ & $\frac{1}{2}$ & $\frac{33}{2}$ & $\frac{231}{2}$ & $\frac{925}{2}
			$ & $\frac{5505}{4}$ & $\frac{13587}{4}$ & $7364$ & $14508$ & $26550\cdots $
			\\ 
			$7$ & $0$ & $\frac{1}{2}$ & $\frac{65}{2}$ & $\frac{1287}{4}$ & $1619$ & $%
			5725$ & $\frac{65007}{4}$ & $\frac{317275}{8}$ & $86608$ & $173664\cdots $
			\\ 
			$8$ & $0$ & $\frac{1}{2}$ & $\frac{129}{2}$ & $\frac{1833}{2}$ & $\frac{11665%
			}{2}$ & $\frac{49155}{2}$ & $\frac{160671}{2}$ & $\frac{441469}{2}$ & $\frac{%
			1068453}{2}$ & $1173240\cdots $ \\ 
		$\underset{\vdots }{9}$ & $0$ & $\frac{1}{2}$ & $\frac{513}{2}$ & $\frac{%
			15531}{2}$ & $\frac{161365}{2}$ & $\frac{1951155}{4}$ & $\frac{8499057}{4}$
		& $7418789$ & $22071123$ & $\frac{232549335}{4}\cdots $%
	\end{tabular}%
\end{equation*}
\end{remark}

By using the numbers $y_{3}\left( n,k;\lambda ;a,b\right) $, we can compute
the second kind Apostol type Euler polynomials and numbers of order $-k$.

Substituting $t=2z$, $a=\frac{x+k}{2k}$, $b=\frac{x-k}{2k}$ into (\ref{St2a1}%
), and replacing $k$ by $-k$ in (\ref{Eul.2}), we get%
\begin{equation*}
	F_{P}(z,x;-k,\lambda )=\frac{k!}{2^{k}\lambda ^{k}}F_{y_{3}}\left(
	2z,k;\lambda ^{2};\frac{x+k}{2k},\frac{x-k}{2k}\right) .
\end{equation*}%
From this equation we get%
\begin{equation*}
	\sum_{n=0}^{\infty }E_{n}^{\ast (-k)}(x;\lambda )\frac{t^{n}}{n!}=\frac{k!}{
		2^{k}\lambda ^{k}}\sum_{n=0}^{\infty }2^{n}y_{3}\left( n,k;\lambda ^{2};%
	\frac{x+k}{2k},\frac{x-k}{2k}\right) \frac{t^{n}}{n!}.
\end{equation*}%
Comparing the coefficients of $\frac{t^{n}}{n!}$ on both sides of the above
equation, we arrive at the following theorem.

\begin{theorem}
	Let $n$ be nonnegative integers. Then we have 
	\begin{equation}
		E_{n}^{\ast (-k)}(x;\lambda )=\frac{k!2^{n-k}}{\lambda ^{k}}y_{3}\left(
		n,k;\lambda ^{2};\frac{x+k}{2k},\frac{x-k}{2k}\right) .  \label{Eul.3c}
	\end{equation}
\end{theorem}

The next assertion confirms and extend a formula in Equation (\ref{Caa3}).

Substituting $x=0$ into (\ref{Eul.3c}), we get the following corollary:

\begin{corollary}
	Let $n$ be nonnegative integers. Then we have 
	\begin{equation}
		E_{n}^{\ast (-k)}(\lambda )=\frac{k!2^{n-k}}{\lambda ^{k}}y_{3}\left(
		n,k;\lambda ^{2};\frac{1}{2},-\frac{1}{2}\right) .  \label{Eul.3d}
	\end{equation}
\end{corollary}

\begin{remark}
	Substituting $\lambda =1$ into (\ref{Eul.3d}), we have%
	\begin{equation*}
		E_{n}^{\ast (-k)}=k!2^{n-k}y_{3}\left( n,k;1;\frac{1}{2},-\frac{1}{2}\right)
	\end{equation*}%
	(\textit{cf}. \cite{Liu1}, \cite{SimsekNEW}, \cite{SrivastavaChoi2012}; see
	also the references cited in each of these earlier works).
\end{remark}

We are finally ready to compute the some values of the second kind Apostol
type Euler polynomials of negative order. By using (\ref{St2a}) and (\ref%
{Eul.3c}), we compute a few values of the polynomials $E_{n}^{\ast
	(-k)}(x;\lambda )$ as follows:%
\begin{eqnarray*}
	E_{n}^{\ast (0)}(x;\lambda ) &=&{x}^{n}, \\
	E_{n}^{\ast (-1)}(x;\lambda ) &=&\frac{{\left( x+1\right) }^{n}\lambda + 
		\frac{{\left( x-1\right) }^{n}}{\lambda }}{2}, \\
	E_{n}^{\ast (-2)}(x;\lambda ) &=&\frac{{\left( x+2\right) }^{n}\lambda ^{2}+ 
		\frac{{\left( x-2\right) }^{n}}{\lambda ^{2}}+2{x}^{n}}{4}, \\
	E_{n}^{\ast (-3)}(x;\lambda ) &=&\frac{{\left( x+3\right) }^{n}\lambda
		^{3}+3 {\left( x+1\right) }^{n}\lambda +\frac{3{\left( x-1\right) }^{n}}{%
			\lambda }+ \frac{{\left( x-3\right) }^{n}}{\lambda ^{3}}}{8}, \\
	E_{n}^{\ast (-4)}(x;\lambda ) &=&\frac{{\left( x+4\right) }^{n}\lambda
		^{4}+4 {\left( x+2\right) }^{n}\lambda ^{2}+\frac{4{\left( x-2\right) }^{n}}{
			\lambda ^{2}}+\frac{{\left( x-4\right) }^{n}}{\lambda ^{4}}+6{x}^{n}}{16}, \\
	E_{n}^{\ast (-5)}(x;\lambda ) &=&\frac{{\left( x+5\right) }^{n}\lambda
		^{5}+5 {\left( x+3\right) }^{n}\lambda ^{3}+10{\left( x+1\right) }%
		^{n}\lambda + \frac{10{\left( x-1\right) }^{n}}{\lambda }+\frac{5{\left(
				x-3\right) }^{n}}{ \lambda ^{3}}+\frac{{\left( x-5\right) }^{n}}{\lambda ^{5}%
		}}{32}
	\end{eqnarray*}%
	Substituting $\lambda =1$ into the above table, we have%
	\begin{eqnarray*}
		E_{n}^{\ast (0)}(x) &=&{x}^{n}, \\
		E_{n}^{\ast (-1)}(x) &=&\frac{{\left( x+1\right) }^{n}+{\left( x-1\right) }
			^{n}}{2}, \\
		E_{n}^{\ast (-2)}(x) &=&\frac{{\left( x+2\right) }^{n}+2{x}^{n}+{\left(
				x-2\right) }^{n}}{4}, \\
		E_{n}^{\ast (-3)}(x) &=&\frac{{\left( x+3\right) }^{n}+3{\left( x+1\right) }
			^{n}+3{\left( x-1\right) }^{n}+{\left( x-3\right) }^{n}}{8}, \\
		E_{n}^{\ast (-4)}(x) &=&\frac{{\left( x+4\right) }^{n}+4{\left( x+2\right) }
			^{n}+6{x}^{n}+4{\left( x-2\right) }^{n}+{\left( x-4\right) }^{n}}{16}, \\
		E_{n}^{\ast (-5)}(x) &=&\frac{{\left( x+5\right) }^{n}+5\,{\left( x+3\right) 
			}^{n}+10{\left( x+1\right) }^{n}+10{\left( x-1\right) }^{n}+5{\left(
			x-3\right) }^{n}+{\left( x-5\right) }^{n}}{32}
\end{eqnarray*}%
By using (\ref{St2a}) and (\ref{Eul.3d}), we compute a few values of the
polynomials $E_{n}^{\ast (-k)}(\lambda )$ as follows:%
\begin{eqnarray*}
	E_{n}^{\ast (0)}(\lambda ) &=&1, \\
	E_{n}^{\ast (-1)}(\lambda ) &=&\frac{\lambda +\frac{{\left( -1\right) }^{n}}{
			\lambda }}{2}, \\
	E_{n}^{\ast (-2)}(\lambda ) &=&\frac{{2}^{n}\lambda ^{2}+\frac{{\left(
				-2\right) }^{n}}{\lambda ^{2}}}{4}, \\
	E_{n}^{\ast (-3)}(\lambda ) &=&\frac{{3}^{n}\lambda ^{3}+3\lambda +\frac{3\,{%
				\ \left( -1\right) }^{n}}{\lambda }+\frac{{\left( -3\right) }^{n}}{\lambda
			^{3} }}{8}, \\
	E_{n}^{\ast (-4)}(\lambda ) &=&\frac{{4}^{n}\lambda ^{4}+{2}^{n+2}\lambda
		^{2}+\frac{{\left( -2\right) }^{n+2}}{\lambda ^{2}}+\frac{{\left( -4\right) }
			^{n}}{\lambda ^{4}}}{16}, \\
	E_{n}^{\ast (-5)}(\lambda ) &=&\frac{{5}^{n}\lambda ^{5}+5.{3}^{n}\lambda
		^{3}+10\lambda +\frac{10{\left( -1\right) }^{n}}{\lambda }+\frac{5{\left(
				-3\right) }^{n}}{\lambda ^{3}}+\frac{{\left( -5\right) }^{n}}{\lambda ^{5}}}{
		32}
\end{eqnarray*}
Substituting $\lambda =1$ into the above table, we have a few values of the
second kind Apostol Euler numbers of order $-k$ as follows: 
\begin{eqnarray*}
	E_{n}^{\ast (0)} &=&1, \\
	E_{n}^{\ast (-1)} &=&\frac{{\left( -1\right) }^{n}+1}{2}, \\
	E_{n}^{\ast (-2)} &=&\frac{{2}^{n}+{\left( -2\right) }^{n}}{4}, \\
	E_{n}^{\ast (-3)} &=&\frac{{3}^{n}+3{\left( -1\right) }^{n}+{\left(
			-3\right) }^{n}+3}{8}, \\
	E_{n}^{\ast (-4)} &=&\frac{{2}^{n+2}+{\left( -2\right) }^{n+2}+{4}^{n}+{\
			\left( -4\right) }^{n}}{16}, \\
	E_{n}^{\ast (-5)} &=&\frac{{5}^{n}+5.{3}^{n}+10{\left( -1\right) }^{n}+5{\
			\left( -3\right) }^{n}+{\left( -5\right) }^{n}+10}{32}
\end{eqnarray*}

When we look carefully the above table and the generating functions, we can
easily get the following result: 
\begin{equation*}
	E_{2n+1}^{\ast (-k)}=0,
\end{equation*}
where $n\geq 0$.

That is for $n=1,2,\ldots ,9$ and $k=0,-1,-2,\ldots ,-9$, we compute a few
values of the numbers $E_{n}^{\ast (-k)}$, given by the above relations, as
follows: 
\begin{equation*}
	\begin{tabular}{lllllllllll}
		$n\backslash k$ & $0$ & $-1$ & $-2$ & $-3$ & $-4$ & $-5$ & $-6$ & $-7$ & $-8$
		& $-9\cdots $ \\ 
		$0$ & $1$ & $0$ & $1$ & $0$ & $1$ & $0$ & $1$ & $0$ & $1$ & $0$ \\ 
		$1$ & $0$ & $0$ & $0$ & $0$ & $0$ & $0$ & $0$ & $0$ & $0$ & $0$ \\ 
		$2$ & $0$ & $1$ & $2$ & $3$ & $4$ & $5$ & $6$ & $7$ & $8$ & $9$ \\ 
		$3$ & $0$ & $0$ & $0$ & $0$ & $0$ & $0$ & $0$ & $0$ & $0$ & $0$ \\ 
		$4$ & $0$ & $1$ & $8$ & $21$ & $40$ & $65$ & $96$ & $133$ & $176$ & $225$ \\ 
		$5$ & $0$ & $0$ & $0$ & $0$ & $0$ & $0$ & $0$ & $0$ & $0$ & $0$ \\ 
		$6$ & $0$ & $1$ & $32$ & $183$ & $544$ & $1205$ & $2256$ & $3787$ & $5888$ & 
		$8649$ \\ 
		$7$ & $0$ & $0$ & $0$ & $0$ & $0$ & $0$ & $0$ & $0$ & $0$ & $0$ \\ 
		$8$ & $0$ & $1$ & $128$ & $1641$ & $8320$ & $26465$ & $64896$ & $134953$ & $%
		250496$ & $427905$ \\ 
		$\underset{\vdots }{9}$ & $0$ & $0$ & $0$ & $0$ & $0$ & $0$ & $0$ & $0$ & $0$
		& $0\cdots $%
	\end{tabular}%
\end{equation*}

\subsection{Algorithm for our computations}

The theory of the algorithms has been very important in Mathematics and in
Computer Science and also in Communications Systems. We know that there are
many ways to compute the second kind Euler polynomials of the negative
order. In this section we give a computation algorithm for computing the
values of these polynomials, which are given by Equations (\ref{St2a}) and ( %
\ref{Eul.3c}).
\section{New families of numbers and polynomials}

In this section, we investigate some properties of the .numbers $%
W_{n}(\lambda )$, which are related to the the second kind Apostol type
Euler polynomials of order $2$, $E_{n}^{\ast (2)}(1;\lambda )$.

For $n$ and $k$ nonnegative integers, $W_{n}^{(k)}(\lambda )$ define by
means of the following genearting function%
\begin{equation}
	F_{w}(t;\lambda ;k)=\frac{1}{\left( \lambda e^{t}+\lambda
		^{-1}e^{-t}+2\right) ^{k}}=\sum_{n=0}^{\infty }W_{n}^{(k)}(\lambda )\frac{%
		t^{n}}{n!}.  \label{w1a}
\end{equation}

By using the Umbral calculus convention in (\ref{w1}), we get a recurrence
relation for the numbers $W_{n}(\lambda )$. Therefore, we set the following
functional equation%
\begin{equation*}
	1=\left( \lambda e^{t}+\lambda ^{-1}e^{-t}+2\right) \sum_{n=0}^{\infty
	}W_{n}(\lambda )\frac{t^{n}}{n!}.
\end{equation*}%
We make some elementary calculations in the above equation, we have%
\begin{equation*}
	1=\lambda \sum_{n=0}^{\infty }\frac{t^{n}}{n!}\sum_{n=0}^{\infty
	}W_{n}(\lambda )\frac{t^{n}}{n!}+\lambda ^{-1}\lambda \sum_{n=0}^{\infty }%
	\frac{\left( -t\right) ^{n}}{n!}\sum_{n=0}^{\infty }W_{n}(\lambda )\frac{%
		t^{n}}{n!}+2\sum_{n=0}^{\infty }W_{n}(\lambda )\frac{t^{n}}{n!}.
\end{equation*}%
By using the Cauchy product and $W^{n}(\lambda )$ is replaced by $%
W_{n}(\lambda )$ in the above equation, we obtain%
\begin{equation*}
	1=\lambda \sum_{n=0}^{\infty }\sum_{m=0}^{n}\left( 
	\begin{array}{c}
		n \\ 
		m%
	\end{array}%
	\right) W_{m}(\lambda )\frac{t^{n}}{n!}+\lambda ^{-1}\sum_{n=0}^{\infty
	}\sum_{m=0}^{n}(-1)^{n-m}\left( 
	\begin{array}{c}
		n \\ 
		m%
	\end{array}%
	\right) W_{m}(\lambda )\frac{t^{n}}{n!}+2\sum_{n=0}^{\infty }W_{n}(\lambda )%
	\frac{t^{n}}{n!}.
\end{equation*}

By comparing the coefficients of $\frac{t^{n}}{n!}$ on both sides of the
above equation, we arrive at the following theorem:

\begin{theorem}
	Let $n$ be a positive integer. and let%
	\begin{equation*}
		W_{0}(\lambda )=\frac{\lambda }{\left( \lambda +1\right) ^{2}}.
	\end{equation*}%
	The following recurrence relation holds true:%
	\begin{equation}
		2W_{n}(\lambda )+\lambda \sum_{m=0}^{n}\left( 
		\begin{array}{c}
			n \\ 
			m%
		\end{array}%
		\right) W_{m}(\lambda )+\lambda ^{-1}\sum_{m=0}^{n}(-1)^{n-m}\left( 
		\begin{array}{c}
			n \\ 
			m%
		\end{array}%
		\right) W_{m}(\lambda )=0.  \label{w2}
	\end{equation}
\end{theorem}

By using (\ref{w2}), we compute a few values of the numbers $W_{n}(\lambda )$
as follows:%
\begin{equation*}
	W_{1}(\lambda )=-\frac{\lambda \left( \lambda -1\right) }{(\lambda +1)^{3}}%
	,W_{2}(\lambda )=-\frac{2\lambda ^{2}}{(\lambda +1)^{4}},W_{3}(\lambda )=%
	\frac{4\lambda (1-\lambda )(\lambda ^{2}-\lambda +1)}{(\lambda +1)^{5}}%
	,\cdots
\end{equation*}

We can show that the numbers $W_{n}(\lambda )$ are associated with the the
second kind Apostol type Euler polynomials of order $2$, $E_{n}^{\ast
	(2)}(1;\lambda )$. By combining (\ref{Eul.2}) with (\ref{w1}), we get%
\begin{equation*}
	\sum_{n=0}^{\infty }W_{n}(\lambda )\frac{t^{n}}{n!}=\sum_{n=0}^{\infty }%
	\frac{\lambda }{4}E_{n}^{\ast (2)}(1;\lambda )\frac{t^{n}}{n!}.
\end{equation*}%
Comparing the coefficients of $\frac{t^{n}}{n!}$ on both sides of the above
equation, we get the following relation:%
\begin{equation*}
	W_{n}(\lambda )=\frac{\lambda }{4}E_{n}^{\ast (2)}(1;\lambda ).
\end{equation*}%
Substituting $k=2$ into (\ref{Eul.2}) and combining with the above equation,
we get the following theorem:

\begin{theorem}
	\begin{equation*}
		W_{n}(\lambda )=\frac{\lambda }{4}\sum_{m=0}^{n}\left( 
		\begin{array}{c}
			n \\ 
			m%
		\end{array}%
		\right) E_{m}^{\ast }(1;\lambda )E_{n-m}^{\ast }(1;\lambda ).
	\end{equation*}
\end{theorem}

By using (\ref{C1}) and (\ref{w1a}), we get the following functional
equation:%
\begin{equation*}
	F_{y_{2}}(t,k;\lambda )F_{w}(t;\lambda ;k)=1.
\end{equation*}%
From this equation, we get%
\begin{equation*}
	\sum_{n=0}^{\infty }W_{n}^{(k)}(\lambda )\frac{t^{n}}{n!}\sum_{n=0}^{\infty
	}y_{2}(n,k;\lambda )\frac{t^{n}}{n!}=1.
\end{equation*}%
Therefore%
\begin{equation*}
	\sum_{n=0}^{\infty }\sum_{m=0}^{n}\left( 
	\begin{array}{c}
		n \\ 
		m%
	\end{array}%
	\right) W_{n-m}^{(k)}(\lambda )y_{2}(m,k;\lambda )\frac{t^{n}}{n!}=1.
\end{equation*}%
Comparing the coefficients of $\frac{t^{n}}{n!}$ on both sides of the above
equation, we give a relation between the numbers $y_{2}(n,k;\lambda )$ and $%
W_{n}^{(k)}(\lambda )$ by the following theorem:

\begin{theorem}
	Let $n$ be a positive integer. Than we have%
	\begin{equation*}
		\sum_{m=0}^{n}\left( 
		\begin{array}{c}
			n \\ 
			m%
		\end{array}%
		\right) W_{n-m}^{(k)}(\lambda )y_{2}(m,k;\lambda )=0.
	\end{equation*}
\end{theorem}

For $x$ real numbers, define $W_{n}^{(k)}(x;\lambda )$ by means of the
following generating function%
\begin{equation}
	G_{w}(t,x,k;\lambda )=e^{tx}F_{w}(t,k;\lambda )=\sum_{n=0}^{\infty
	}W_{n}^{(k)}(x;\lambda )\frac{t^{n}}{n!}.  \label{w1b}
\end{equation}

Combining (\ref{w1a}) with (\ref{w1b}), we get%
\begin{equation*}
	\sum_{n=0}^{\infty }W_{n}^{(k)}(x;\lambda )\frac{t^{n}}{n!}%
	=\sum_{n=0}^{\infty }\frac{t^{n}}{n!}\sum_{n=0}^{\infty }W_{n}^{(k)}(\lambda
	)\frac{t^{n}}{n!}.
\end{equation*}%
Therefore%
\begin{equation*}
	\sum_{n=0}^{\infty }W_{n}^{(k)}(x;\lambda )\frac{t^{n}}{n!}%
	=\sum_{n=0}^{\infty }\sum_{m=0}^{n}\left( 
	\begin{array}{c}
		n \\ 
		m%
	\end{array}%
	\right) x^{n-m}W_{m}^{(k)}(\lambda )\frac{t^{n}}{n!}.
\end{equation*}%
Comparing the coefficients of $\frac{t^{n}}{n!}$ on both sides of the above
equation, we get the following theorem:

\begin{theorem}
	\begin{equation*}
		W_{n}^{(k)}(x;\lambda )=\sum_{m=0}^{n}\left( 
		\begin{array}{c}
			n \\ 
			m%
		\end{array}%
		\right) x^{n-m}W_{m}^{(k)}(\lambda ).
	\end{equation*}
\end{theorem}

By using (\ref{w1b}) and (\ref{Cad3}), we get the following functional
equation:%
\begin{equation*}
	G(t,x,k;\lambda )=\frac{1}{4^{k}}F_{P1}\left( t,\frac{x}{2};k,\lambda
	\right) F_{P1}\left( -t,\frac{x}{2};k,\lambda ^{-1}\right) .
\end{equation*}%
By combining this equation with (\ref{w1b}) and (\ref{Cad3}), we get%
\begin{equation*}
	\sum_{n=0}^{\infty }W_{n}^{(k)}(x;\lambda )\frac{t^{n}}{n!}=\frac{1}{4^{k}}%
	\sum_{n=0}^{\infty }E_{n}^{(k)}\left( \frac{x}{2};\lambda \right) \frac{t^{n}%
	}{n!}\sum_{n=0}^{\infty }(-1)^{n}E_{n}^{(k)}\left( \frac{x}{2};\lambda
	^{-1}\right) \frac{t^{n}}{n!}.
\end{equation*}%
Therefore%
\begin{equation*}
	\sum_{n=0}^{\infty }W_{n}^{(k)}(x;\lambda )\frac{t^{n}}{n!}=\frac{1}{4^{k}}%
	\sum_{n=0}^{\infty }\sum_{m=0}^{n}(-1)^{n-m}\left( 
	\begin{array}{c}
		n \\ 
		m%
	\end{array}%
	\right) E_{m}^{(k)}\left( \frac{x}{2};\lambda \right) E_{n-m}^{(k)}\left( 
	\frac{x}{2};\lambda ^{-1}\right) \frac{t^{n}}{n!}.
\end{equation*}%
Comparing the coefficients of $\frac{t^{n}}{n!}$ on both sides of the above
equation, we get the following theorem:

\begin{theorem}
	\begin{equation}
		W_{n}^{(k)}(x;\lambda )=\frac{1}{4^{k}}\sum_{m=0}^{n}(-1)^{n-m}\left( 
		\begin{array}{c}
			n \\ 
			m%
		\end{array}%
		\right) E_{m}^{(k)}\left( \frac{x}{2};\lambda \right) E_{n-m}^{(k)}\left( 
		\frac{x}{2};\lambda ^{-1}\right) .  \label{w1c}
	\end{equation}
\end{theorem}

Substituting $x=0$ into (\ref{w1c}), we get%
\begin{equation*}
	W_{n}^{(k)}(x;\lambda )=\frac{1}{4^{k}}\sum_{m=0}^{n}(-1)^{n-m}\left( 
	\begin{array}{c}
		n \\ 
		m%
	\end{array}%
	\right) E_{m}^{(k)}\left( \lambda \right) E_{n-m}^{(k)}\left( \lambda
	^{-1}\right) .
\end{equation*}%
Setting $k=1$ in the above equation, we have%
\begin{equation*}
	W_{n}(x;\lambda )=\frac{1}{4}\sum_{m=0}^{n}(-1)^{n-m}\left( 
	\begin{array}{c}
		n \\ 
		m%
	\end{array}%
	\right) E_{m}\left( \lambda \right) E_{n-m}\left( \lambda ^{-1}\right) .
\end{equation*}%
By using (\ref{C1}), (\ref{Cac3}) and (\ref{w1a}), we get the following
functional equations, respectively:%
\begin{equation}
	F_{w}(t;\lambda ;-k)=(2k)!F_{y_{2}}(t,k;\lambda )  \label{1F}
\end{equation}%
and%
\begin{equation}
	F_{N}(t;-k,\lambda )=\sum_{m=0}^{k}\left( 
	\begin{array}{c}
		k \\ 
		m%
	\end{array}%
	\right) \frac{1}{2^{m}}F_{w}(t;\lambda ;-m).  \label{2F}
\end{equation}%
By using (\ref{1F}), the numbers $W_{n}(\lambda )$ of order $-k$, denoted by 
$W_{n}^{(-k)}(\lambda )$, can be computed by the following theorem:

\begin{theorem}
	\begin{equation}
		W_{n}^{(-k)}(\lambda )=(2k)!y_{2}(n,k;\lambda ).  \label{w2A}
	\end{equation}
\end{theorem}

Thanks to Equation (\ref{2F}), a relation between the numbers $E_{n}^{\ast
	(-k)}(\lambda )$ and $W_{n}^{(-k)}(\lambda )$ are given by the following
theorem:

\begin{theorem}
	\begin{equation*}
		E_{n}^{\ast (-k)}(\lambda )=\sum_{m=0}^{k}\left( 
		\begin{array}{c}
			k \\ 
			m%
		\end{array}%
		\right) \frac{1}{2^{m}}W_{n}^{(-m)}(\lambda ).
	\end{equation*}
\end{theorem}

By using (\ref{CCC3}), for $k=0;1;2;3$ and we $n=0;1;2;3;4$, we can compute
a few values of the numbers $W_{n}^{(-k)}(\lambda )$ given by Equation (\ref%
{w2A}) as follows:

$%
\begin{array}{ccccc}
n\backslash k & 0 & 1 & 2 & 3 \\ 
0 & 1 & \lambda +\frac{1}{\lambda } & \lambda ^{2}+4\lambda +\frac{4\lambda
	+1}{\lambda ^{2}} & \lambda ^{3}+6\lambda ^{2}+15\lambda +\frac{\lambda }{15}%
+\frac{6\lambda +1}{\lambda ^{3}} \\ 
1 & 0 & \lambda -\frac{1}{\lambda } & 2\lambda ^{2}+4\lambda -\frac{4\lambda
	+2}{\lambda ^{2}} & 3\lambda ^{3}+12\lambda ^{2}+15\lambda -\frac{\lambda }{%
	15}-\frac{12\lambda +3}{\lambda ^{3}} \\ 
2 & 0 & \lambda +\frac{1}{\lambda } & 4\lambda ^{2}+4\lambda +\frac{4\lambda
	+4}{\lambda ^{2}} & 9\lambda ^{3}+24\lambda ^{2}+15\lambda +\frac{\lambda }{%
	15}+\frac{24}{\lambda ^{2}}+\frac{9}{\lambda ^{3}} \\ 
3 & 0 & \lambda -\frac{1}{\lambda } & 4\lambda ^{2}+2\lambda -\frac{2\lambda
	+4}{\lambda ^{2}} & 27\lambda ^{3}+48\lambda ^{2}+15\lambda -\frac{15}{%
	\lambda }-\frac{48}{\lambda ^{2}}-\frac{27}{\lambda ^{3}} \\ 
4 & 0 & \lambda +\frac{1}{\lambda } & 16\lambda ^{2}+8\lambda +\frac{%
	4\lambda +16}{\lambda ^{2}} & 81\lambda ^{3}+96\lambda ^{2}+15\lambda +\frac{%
	\lambda }{15}+\frac{96}{\lambda ^{2}}+\frac{81}{\lambda ^{3}}%
\end{array}%
$

\section{Further remarks and observations}

In this section we outline some application of our numbers to combinatorial
analysis. In \cite{SimsekNEW}, we gave combinatorial interpretations of the
numbers, $y_{1}(n,k)$, $y_{2}(n,k)$ and the central factorial numbers. We
pursued that the numbers $B(n,k)$ related to the enumerative combinatorics
with the following Bona \cite[P. 46, Exercise 3-4]{Bona} exercises:

\textbf{Exercise 3}. Find the number of ways to place n rooks on an $n\times
n$ chess board so that no two of them attack each other.

\textbf{Exercise 4}. How many ways are there to place some rooks on an $%
n\times n$ chess board so that no two of them attack each other?

As a result, our new special numbers with their generating functions have
many applications in combinatorial analysis and in analytic number theory.
These numbers are related to the many well-known numbers and polynomials:
the Bernstein basis functions, the array polynomials, the Stirling numbers
of the second kind, the central factorial numbers and also the Golombek's
problem \cite{golombek} \textquotedblleft \textbf{Aufgabe 1088 }%
\textquotedblright .

\begin{acknowledgement}
	The paper was supported by the \textit{Scientific Research Project
		Administration of Akdeniz University.}
\end{acknowledgement}


\end{document}